\documentclass[12pt, reqno]{amsart}
\usepackage{amssymb,latexsym,amsmath,amscd,amsthm,graphicx, color}
\usepackage[all]{xy}
\usepackage{pgf,tikz}
\usepackage{mathrsfs}
\usetikzlibrary{arrows}
\usepackage[left=0.6 in, top=0.6 in, right=0.6 in, bottom=0.5 in]{geometry}

\makeatletter
\newcommand{\setword}[2]{%
  \phantomsection
  #1\def\@currentlabel{\unexpanded{#1}}\label{#2}%
}
\makeatother

\usepackage[linkcolor=blue, urlcolor=blue, citecolor=blue,
colorlinks, bookmarks]{hyperref}

\definecolor{uuuuuu}{rgb}{0.26666666666666666,0.26666666666666666,0.26666666666666666}
\definecolor{xdxdff}{rgb}{0.49019607843137253,0.49019607843137253,1.}
\definecolor{ffqqqq}{rgb}{1.,0.,0.}
\definecolor{ffqqqq}{rgb}{1.,0.,0.}
\definecolor{ffxfqq}{rgb}{1.,0.4980392156862745,0.}

\definecolor{uuuuuu}{rgb}{0.26666666666666666,0.26666666666666666,0.26666666666666666}
\definecolor{qqwuqq}{rgb}{0.,0.39215686274509803,0.}
\definecolor{zzttqq}{rgb}{0.6,0.2,0.}
\definecolor{xdxdff}{rgb}{0.49019607843137253,0.49019607843137253,1.}
\definecolor{qqqqff}{rgb}{0.,0.,1.}
\definecolor{cqcqcq}{rgb}{0.7529411764705882,0.7529411764705882,0.7529411764705882}
\definecolor{sqsqsq}{rgb}{0.12549019607843137,0.12549019607843137,0.12549019607843137}

\theoremstyle{plain}

\newtheorem{theorem}[subsection]{Theorem}

\newtheorem{lemma}[subsection]{Lemma}
\newtheorem{defi}[subsection]{Definition}
\newtheorem{prop}[subsection]{Proposition}

\theoremstyle{definition}

\newtheorem{cor}[subsection]{Corollary}

\newtheorem{exam}[subsection]{Example}

\newtheorem{remark}[subsection]{Remark}

\newtheorem{notation}[subsection]{Notation}
\newtheorem{note}[subsection]{Note}

\newcommand{\uu}{\cup}
\newcommand{\ii}{\cap}
\newcommand{\UU}{\bigcup}

\newcommand{\ci}{\subseteq}
\newcommand{\sci}{\subset}

\newcommand{\set}[1]{\{#1\}}

\newcommand{\ga}{\alpha}
\newcommand{\gb}{\beta}

\renewcommand{\gg}{\gamma}

\newcommand{\gk}{\kappa}

\newcommand{\gq}{\theta}

\newcommand{\tit}{\textit}

\newcommand{\D}[1]{\mathbb{#1}}

\newcommand{\te}{\text}

\begin{document}
To appear, in the Journal `Mathematics'
\title{Conditional quantization for uniform distributions on line segments and regular polygons}

\address{School of Mathematical and Statistical Sciences\\
The University of Texas Rio Grande Valley\\
1201 West University Drive\\
Edinburg, TX 78539-2999, USA.}

\email{\{$^1$pigar.biteng01, $^2$mathieu.caguiat01, $^3$tsianna.dominguez01, $^4$mrinal.roychowdhury\}\newline @utrgv.edu}
 
\author{$^1$Pigar Biteng}
\author{$^2$Mathieu Caguiat}
\author{$^3$Tsianna Dominguez}
\author{$^4$Mrinal Kanti Roychowdhury}

\thanks{A part of the results in this paper was used in partial fulfillment of the third author's Master's thesis at the University of Texas Rio Grande Valley under the direction of the last author.}


\date{}
\maketitle

\pagestyle{myheadings}\markboth{P. Biteng, M. Caguiat, T. Dominguez,  M.K. Roychowdhury}{Conditional quantization for uniform distributions on line segments and regular polygons}

\begin{abstract}
 Quantization for a Borel probability measure refers to the idea of estimating a given probability by a discrete probability with support containing a finite number of elements. If in the quantization some of the elements in the support are preselected, then the quantization is called a conditional quantization.   
In this paper, we investigate the conditional quantization for the uniform distributions defined on the unit line segments and $m$-sided regular polygons, where $m\geq 3$, inscribed in a unit circle.  
\end{abstract}

\section{Introduction}

The process of transformation of a continuous-valued signal into a discrete-valued one is called `quantization'. It has broad applications in engineering and technology.  We refer to \cite{GG, GN, Z2} for surveys on the subject and comprehensive lists of references to the literature; see also \cite{AW, GKL, GL1, Z1}. 
For mathematical treatment of quantization one is referred to Graf-Luschgy's book (see \cite{GL1}). For some other recent papers on quantization one can see \cite{DFG, DR, GG, GL, GL1, GL2, GL3, GN, KNZ, P, P1, R1, R2, R3, Z1, Z2}.
Recently, Pandey and Roychowdhury introduced the concepts of constrained quantization and the conditional quantization (for example, see \cite{BCDRV, PR1, PR2, PR4}).
This paper deals with conditional quantization. 

 \begin{defi}  \label{defi0} 
Let $P$ be a Borel probability measure on $\D R^2$ equipped with a Euclidean metric $d$ induced by the Euclidean norm $\|\cdot\|$.  Let $\gb\sci \D R^2$ be given with $\te{card}(\gb)=\ell$ for some $\ell\in \D N$. 
Then, for $n\in \D N$ with $n\geq \ell$, the \tit {$n$th conditional quantization
error} for $P$ with respect to the conditional set $\gb$, is defined as
\begin{equation}  \label{eq0} 
V_{n}:=V_{n}(P)=\inf_{\ga} \Big\{\int \mathop{\min}\limits_{a\in\ga\uu\gb} d(x, a)^2 dP(x) : \text{card}(\ga) \leq n-\ell \Big\},
\end{equation} 
where $\te{card}(A)$ represents the cardinality of the set $A$. 
\end{defi} 
We assume that $\int d(x, 0)^2 dP(x)<\infty$ to make sure that the infimum in \eqref{eq0} exists (see \cite{PR1}).  For a finite set $\gg \sci \D R^2$ and $a\in \gg$, by $M(a|\gg)$ we denote the set of all elements in $\D R^2$ which are nearest to $a$ among all the elements in $\gg$, i.e.,
$M(a|\gg)=\set{x \in \D R^2 : d(x, a)=\mathop{\min}\limits_{b \in \gg}d(x, b)}.$
$M(a|\gg)$ is called the \tit{Voronoi region} in $\D R^2$ generated by $a\in \gg$. 

\begin{defi}
A set $ \ga\uu\gb$, where $P(M(b|\ga\uu \gb))>0$ for $b\in \gb$, for which the infimum in $V_n$ exists and contains no less than $\ell$ elements, and no more than $n$ elements is called a \tit{conditional optimal set of $n$-points} for $P$ with respect to the conditional set $\gb$.
\end{defi}
Let $V_{n, r}(P)$ be a strictly decreasing sequence, and write $V_{\infty, r}(P):=\mathop{\lim}\limits_{n\to \infty} V_{n, r}(P)$.  Then, the number 
\begin{align*}
D(P):=\mathop{\lim}\limits_{n\to \infty}  \frac{2\log n}{-\log (V_{n}(P)-V_{\infty}(P))}
\end{align*} 
if it exists, 
 is called the \tit{conditional quantization dimension} of $P$ and is denoted by $D(P)$. The conditional quantization dimension measures the speed at which the specified measure of the conditional quantization error converges as $n$ tends to infinity.
For any $\gk>0$, the number 
\[\lim_{n\to \infty} n^{\frac 2 \gk}  (V_{n}(P)-V_{\infty}(P)),\]
if it exists, is called the \tit{$\gk$-dimensional conditional quantization coefficient} for $P$.  
\par 
In this paper, we have investigated the conditional quantization for uniform distributions on the unit line segments and on regular $m$-sided polygons, where $m\geq 3$, inscribed in a unit circle.

\subsection{Delineation.} In this paper, there are total three sections in addition to the section that contains the basic preliminaries. First, we have proved a proposition Proposition~\ref{prop0}. In Section~\ref{sec3} as a special case of Proposition~\ref{prop0}, we explicitly determine the conditional optimal sets of $n$-points and the $n$th conditional quantization errors for a uniform distribution  with two interior elements as the conditional set for all $n\geq 2$ on the interval $[0, 1]$. In Section~\ref{sec4} as an extension of Proposition~\ref{prop0}, we calculate the conditional optimal sets of $n$-points and the $n$th conditional quantization errors for $(k-1)$ uniformly distributed interior elements on the interval $[0, 1]$. On the other hand, Section~\ref{sec5} is an application of Proposition~\ref{prop0}. It deals with a uniform distribution defined on the boundary of a regular $m$-sided polygon.  Let $P$ be a uniform distribution defined on the boundary of a regular $m$-sided polygon inscribed in a unit circle. 
After the introduction of conditional quantization, we know that the quantization dimension and the quantization coefficient do not depend on the conditional set  (see \cite{PR4}). Using this scenario, in Section~\ref{sec5}, we calculate the quantization coefficient for the uniform distribution $P$ defined on the boundary of the regular $m$-sided polygon inscribed in the unit circle by calculating the conditional quantization coefficient for $P$ with respect to the conditional set $\gb$, which consists of all the vertices of the regular polygon.  In addition, we also give an explicit formula to calculate the conditional optimal  sets of $n$-points and the $n$th conditional quantization errors for the uniform distribution $P$ for all $n\geq m$, where $m$ is the number of vertices of the $m$-sided polygon.

\subsection{Motivation and significance.} Conditional quantization has recently been introduced by Pandey-Roychowdhury in \cite{PR4}. It has significant interdisciplinary applications: for example, in radiation therapy of cancer treatment to find the optimal locations of $n$ centers of radiation, where $k$ centers for some $k<n$ of radiation are preselected, the conditional quantization technique can be used. There are many interesting open problems that can be investigated. The work in this paper is an advancement in this direction. In the paper \cite{HMRT}, when there is no conditional set, Hansen et al. in a proposition, first determined the optimal sets of $n$-means and the $n$th quantization errors for the probability distribution $P$ defined on the boundary of a regular $m$-sided polygon, when $n$ is of the form $n=mk$ for some $k\in \D N$. Then, with the help of the proposition, they showed that the quantization coefficient for $P$ exists, and equals $\frac{1}{3} m^2 \sin ^2\frac{\pi }{m}$, i.e.,
\[\lim_{n\to \infty} n^2 V_n(P)=\frac{1}{3} m^2 \sin ^2\frac{\pi }{m}.\]
In this paper, we have also calculated the quantization coefficient for the same uniform distribution $P$, but the work in this paper is much more simpler than the work to calculate the quantization coefficient done by Hansen et al. in the paper \cite{HMRT}.

\section{Preliminaries} \label{sec0}  
 For any two elements $(a, b)$ and $(c, d)$ in $\D R^2$, we write 
 \[\rho((a, b), (c, d)):=(a-c)^2 +(b-d)^2,\] which gives the squared Euclidean distance between the two elements $(a, b)$ and $(c, d)$.
  Two elements $p$ and $q$ in an optimal set of $n$-points are called \tit{adjacent elements} if they have a common boundary in their own Voronoi regions. Let $e$ be an element on the common boundary of the Voronoi regions of two adjacent elements $p$ and $q$ in an optimal set of $n$-points. Since the common boundary of the Voronoi regions of any two elements is the perpendicular bisector of the line segment joining the elements, we have
\[\rho(p, e)-\rho(q, e)=0. \]
We call such an equation a \tit{canonical equation}. 
Notice that any element $x\in \D R$ can be identified as an element $(x, 0)\in \D R^2$. Thus, 
\[\rho: \D R \times \D R^2 \to [0, \infty) \te{ such that } \rho(x, (a, b))=(x-a)^2 +b^2,\]
where $x\in \D R$ and $(a, b) \in \D R^2$, defines a nonnegative real-valued function on $\D R \times \D R^2$. On the other hand, 
\[\rho: \D R \times \D R  \to [0, \infty) \te{ be such that } \rho(x, y)=(x-y)^2,\]
where $x,y\in \D R$, defines a nonnegative real-valued function on $\D R \times \D R$. 

Let $P$ be a Borel probability measure on $\D R$ which is uniform on its support the closed interval $[a, b]$. Then, the probability density function $f$ for $P$ is given by 
\[f(x)=\left\{\begin{array}{cc}
 \frac 1 {b-a} & \te{ if } a\leq x\leq b,\\
 0 & \te{ otherwise}.
\end{array}\right.
\]
Hence, we have $dP(x)=P(dx)=f(x) dx$ for any $x\in \D R$, where $d$ stands for differential.  
   
\begin{notation}
Let $\ga$ be a discrete set. Then, for a Borel probability measure $\mu$ and a set $A$, by $V(\mu; \set{\ga, A})$, it is meant the distortion error for $\mu$ with respect to the set $\ga$ over the set $A$, i.e., 
\begin{equation} \label{eqMe1} 
  V(\mu; \set{\ga, A}):= \int_A \mathop{\min}\limits_{a\in\ga} \rho(x, a)^2 \,d\mu(x).
\end{equation}
\end{notation}

 The following proposition is a generalized version of Proposition~2.1, Proposition~2.2 and Proposition~2.3 that appear in \cite{PR4}. 

\begin{prop} \label{prop0} 
Let $P$ be a uniform distribution on the closed interval $[a, b]$ and $c, d\in [a, b]$ be such that $a<c<d<b$. For $n\in \D N$ with $n\geq 2$, let $\ga_n$ be a conditional optimal set of $n$-points for $P$ with respect to the conditional set $\gb=\set{c, d}$ such that $\ga_n$ contains $k$ elements from the closed interval $[a, c]$, $\ell$ elements from the closed interval $[c, d],$ and $m$ elements from the closed interval $[d, b]$ for some $k,\ell,  m\in \D N$ with $k, m\geq 1$ and $\ell\geq 2$. Then, $k+\ell+m=n+2$,
\begin{align*} \ga_n\ii [a, c]&= \Big\{a+\frac{(2j-1)(c-a)}{2k-1} : 1\leq j\leq k\Big\}, \\
 \ga_n\ii[c, d]&=\Big\{c+\frac{j-1}{\ell-1}(d-c) : 1\leq j\leq \ell\Big\},  \te{ and } \\
 \ga_n\ii [d, b]&=\Big\{d+\frac{2(j-1)(b-d)}{2m-1} : 1\leq j\leq m\Big\} 
 \end{align*} 
with the conditional quantization error 
 \[V_n:=V_{k,\ell, m}(P)= \frac  1 {3(b-a)}\Big(\frac{(c-a)^3}{(2k-1)^2}+\frac 1 {4}\frac {(d-c)^3}{ (\ell-1)^2}+\frac{(b-d)^3}{(2m-1)^2}\Big).\]
 \end{prop} 
 
  \begin{proof}
 Notice that the element $c$ in the conditional set $\gb$ is common to both the intervals $[a, c]$ and $[c, d]$, the element $d$ in the conditional set $\gb$ is common to both the intervals $[c, d]$ and $[d, b]$, and so $c$ and $d$ are counted two times. Hence, $k+\ell+m=n+2$. 
We have
\[[a, b]:=\set{t : a\leq t\leq b}.\]
Let $\ga_n$ be a conditional optimal set of $n$-points such that 
\[\te{card}(\ga_n\ii [a, c])=k, \  \te{card}(\ga_n\ii [c, d])=\ell, \te{ and } \te{card}(\ga_n\ii [d, b])=m, \te{ where } k, m\geq 1 \te{ and } \ell\geq 2.\]
Then, we can write 
\[\ga_n\ii [a, c]=\set{a_1, a_2,  \cdots, a_k}, \ \ga_n\ii [c, d]=\set{c_1, c_2,  \cdots, c_\ell} \te{ and } \ga_n\ii [d, b]=\set{d_1, d_2, \cdots, d_m},\]
such that 
\[a<a_1<a_2<\cdots <a_k=c=c_1<c_2<\cdots <c_\ell=d=d_1<d_2<\cdots <d_m<b.\]  

We now prove the following claim. 

\tit{Claim.  $a_2-a_1=a_3-a_2=\cdots=a_{k}-a_{k-1}=\frac{a_k-a_1}{k-1}=\frac {c-a_1}{k-1}.$}

Since there is no restriction on the locations of the elements $a_j$ for $1\leq j\leq k-1$, they must be the conditional expectations in their own Voronoi regions. Hence, we have   
\begin{align}
a_1&=E(X : X \in [a, \frac{1}{2} (a_1+a_2)]), \label{eq01} \\
a_i&=E(X : X \in [\frac{a_{i-1}+a_i}{2}, \frac{a_i+a_{i+1}}{2}]) \te{ for } 2\leq i\leq k-2,  \label{eq21}\\
a_{k-1}&=E(X : X\in [\frac{a_{k-2}+a_{k-1}}{2}, c]). \label{eq3}
\end{align}
By \eqref{eq01}, we have
\[a_1=\frac{\int_{a}^{ \frac{1}{2} (a_1+a_2)} x dP}{\int_{a}^{\frac{1}{2} (a_1+a_2)} dP}=\frac{\int_{a}^{\frac{1}{2} (a_1+a_2)} x f(x) dx}{\int_{a}^{ \frac{1}{2} (a_1+a_2)} f(x) dx}=\frac{1}{4} \left(2 a+a_1+a_2\right).\]
Similarly, by \eqref{eq21} for  $2\leq i\leq k-2$, we have 
\[a_i=\frac 14(a_{i-1}+2a_i+a_{i+1}),\]
and by \eqref{eq3}, we deduce 
\[a_{k-1}=\frac 14(a_{k-2}+2a_{k-1}+a_{k}).\]
Combining all the expressions for $a_j$ for $1\leq j\leq k$, we have 
\begin{equation} \label{eq400} a_2-a_1=a_3-a_2=\cdots=a_{k}-a_{k-1}=\frac{a_k-a_1}{k-1}=\frac {c-a_1}{k-1}.
\end{equation} 
Thus, the claim is true. Now, by \eqref{eq400}, we have 
\begin{align*}
a_2&=a_1+\frac {c-a_1}{k-1}=a_1+\frac {c-a_1}{k-1},\\
a_3&=a_2+\frac {c-a_1}{k-1}=a_1+2\frac {c-a_1}{k-1}, \\
a_4&=a_3+\frac {c-a_1}{k-1}=a_1+3\frac {c-a_1}{k-1},\\
&\te{and so on.}
\end{align*} 
Thus, we have $a_j= a_1+(j-1)\frac {c-a_1}{k-1}$ for $ 1\leq j\leq k$. The distortion error due to the elements $\ga_n\ii [a, c]$ is given by 
\begin{align*}
&V(P; \set{\ga_n\ii [a, c], [a, c]})=\int_{[a, c]}\min_{x\in \ga_n\ii [a, c]} \rho(t, x)\, dP\\
&=\frac 1{b-a} \Big( \int_{a}^{\frac  {a_1+a_2}2} \rho(t, a_1)\,dt+(k-2) \int_{\frac  {a_1+a_2}2}^{\frac  {a_2+a_3}2} \rho(t, a_2)\,dt  +\int_{\frac  {a_{k-1}+a_k}2}^{a_k} \rho(t, a_k)\,dt\Big)\\
&=\frac{4 a^3 (k-1)^2-3 a_1 \left(4 a^2 (k-1)^2-c^2\right)+3 a_1^2 \left(4 a (k-1)^2-c\right)-c^3+a_1^3 \left(-4 k^2+8 k-3\right)}{12 (k-1)^2 (a-b)},
\end{align*} 
the minimum value of which is $\frac{(c-a)^3}{3 (b-a)(2 k-1)^2}$ and it occurs when $a_1=a+\frac{c-a}{2 k-1}$. Putting the values of $a_1$, we have
\[a_j=a+\frac{(2j-1)(c-a)}{2k-1} \te{ for } 1\leq j\leq k 
\te{ with } 
 V(P; \set{\ga_n\ii [a, c], [a, c]})=\frac{(c-a)^3}{3 (b-a)(2 k-1)^2}.\]
 Since the closed interval $[c, d]$ is a line segment and $P$ is a uniform distribution, proceeding in the similar way as the proof given in the above claim, we have
\[c_2-c_1=c_3-c_2=\cdots=c_{\ell}-c_{\ell-1}=\frac{c_\ell-c_1}{\ell-1}=\frac {d-c}{\ell-1}\]
implying 
\begin{align*}
c_2&=c_1+\frac {d-c}{\ell-1}=c+\frac {d-c}{\ell-1},\\
c_3&=c_2+\frac {d-c}{\ell-1}=c+\frac {2(d-c)}{\ell-1}, \\
c_4&=c_3+\frac {d-c}{\ell-1}=c+\frac {3(d-c)}{\ell-1},\\
&\te{and so on.}
\end{align*} 
Thus, we have $c_j= c+\frac{j-1}{\ell-1}(d-c)$ for $ 1\leq j\leq \ell$. The distortion error contributed by the $\ell$ elements in the closed interval  $[c, d]$  is given by
 \begin{align*}
V(P; \set{\ga_n\ii [c, d], [c, d]})&=\int_{[c,  d]} \min_{x\in \ga_n\ii[c, d]} \rho((t, 0), x)\, dP\\
&=\frac 1{b-a} \Big(2\int_{c_1}^{\frac  {c_1+c_2}2} \rho((t, 0), (c_1,0))\,dt+(\ell-2) \int_{\frac  {c_1+c_2}2}^{\frac  {c_2+c_3}2} \rho((t, 0), (c_2,0))\,dt\Big)\\
&=\frac 1 {12}\frac{(d-c)^3}{b-a}\frac 1{ (\ell-1)^2}.
\end{align*} 
Again, the closed interval $[d, b]$ is a line segment and $P$ is a uniform distribution. Proceeding in the similar way as the proof given in the above claim, we have 
\[d_2-d_1=d_3-d_2=\cdots=d_{m}-d_{m-1}=\frac{d_m-d_1}{m-1}=\frac {d_m-d}{m-1}\]
implying 
\begin{align*}
d_2&=d_1+\frac {d_m-d}{m-1}=d+\frac {d_m-d}{m-1},\\
d_3&=d_2+\frac {d_m-d}{m-1}=d+2\frac {d_m-d}{m-1}, \\
d_4&=d_3+\frac {d_m-d}{m-1}=d+3\frac {d_m-d}{m-1},\\
&\te{and so on.}
\end{align*} 
Thus, we have $d_j= d+(j-1)\frac {d_m-d}{m-1}$ for $ 1\leq j\leq m$.
The distortion error contributed by the $m$ elements is given by
 \begin{align*}
&V(P; \set{\ga_n\ii [d, b], [d, b]})=\int_{[d, b]}\min_{x\in \ga_n\ii[d, b]} \rho(t, x)\, dP\\
&=\frac 1{b-a} \Big( \int_{d_1}^{\frac  {d_1+d_2}2} \rho(t, d_1)\,dt+(m-2) \int_{\frac  {d_1+d_2}2}^{\frac  {d_2+d_3}2} \rho(t, d_2)\,dt +\int_{\frac  {d_{m-1}+d_m}2}^{b} \rho(t, d_m)\,dt\Big)\\
&=\frac{-4 b^3 (m-1)^2+3 d_m \left(4 b^2 (m-1)^2-d^2\right)-3 d_m^2 \left(4 b (m-1)^2-d\right)+d^3+\left(4 m^2-8 m+3\right) d_m^3}{12 (m-1)^2 (a-b)} 
\end{align*} 
the minimum value of which is $\frac{(b-d)^3}{3(b-a) (2 m-1)^2}$ and it occurs when $d_m=d+\frac{2(m-1)(b-d)}{2 m-1}$. Putting the values of $d_m$, we have
\[d_j=d+\frac{2(j-1)(b-d)}{2m-1} \te{ for } 1\leq j\leq m 
\te{ with }  
V(P; \set{\ga_n\ii [d, b], [d, b]})=\frac{(b-d)^3}{3(b-a) (2 m-1)^2}.\]
 Since $a_j=a+\frac{(2j-1)(c-a)}{2k-1}$ for $1\leq j\leq k$, \ $c_j= c+\frac{j-1}{\ell-1}(d-c)$ for $ 1\leq j\leq \ell$,  and $d_j=d+\frac{2(j-1)(b-c)}{2m-1} \te{ for } 1\leq j\leq m$, and
 \[V_n:=V_{k, \ell, m}=V(P; \set{\ga_n\ii [a, c], [a, c]})+V(P; \set{\ga_n\ii [c, d], [c, d]})+V(P; \set{\ga_n\ii [d, b], [d, b]}),\]
the proposition is yielded. 
  \end{proof}
 
In the following sections, we give the main results of the paper.

\section{Conditional optimal sets of $n$-points and the conditional quantization errors with two interior elements in the conditional set for all $n\geq 2$ on a unit line segment} \label{sec3} 

In this section, for the uniform distribution $P$ on the line segment $[0, 1]$ with respect to the conditional set $\gb:=\set{\frac 14, \frac 12}$, we calculate the conditional optimal sets of $n$-points and the $n$th conditional quantization errors for all $n\in \D N$ with $n\geq 2$. Let $\ga_n$ be a conditional optimal set of $n$-points with the $n$th conditional quantization error $V_n$ for all $n\in \D N$.
Let $\te{card}(\ga_n\ii[0, \frac 14])=k$, $\te{card}(\ga_n\ii [\frac 14, \frac 12])=\ell$, and $\te{card}(\ga_n\ii [\frac 12, 1])=m$. Then, $k, m\geq 1$, and $\ell\geq 2$. By Proposition~\ref{prop0}, we know that
\begin{align}\label{eqMe1}  \ga_n\ii [0, \frac 14]&= \Big\{\frac{2j-1}{4(2k-1)} : 1\leq j\leq k\Big\}, \notag \\
 \ga_n\ii[\frac 14, \frac 12]&=\Big\{\frac 14+\frac{j-1}{4(\ell-1)}  : 1\leq j\leq \ell\Big\},  \te{ and } \\
 \ga_n\ii [\frac 12, 1]&=\Big\{\frac 12+\frac{j-1}{2m-1} : 1\leq j\leq m\Big\} \notag. 
 \end{align}  
Notice that $\ga_n=(\ga_n \ii [0, \frac 14])\uu (\ga_n\ii[\frac 14, \frac 12])\uu (\ga_n\ii [\frac 12, 1])$ with the $n$th conditional quantization error 
\begin{equation} \label{eqMe2} V_n:=V_{k,\ell, m}(P)= \frac 13\Big(\frac{1}{64 (2 k-1)^2}+\frac{1}{256 (\ell-1)^2}+\frac{1}{8 (2 m-1)^2}\Big).
\end{equation} 
\begin{prop}
The optimal set of two-points is the set $\gb=\set{\frac{1}{4},\frac{1}{2}}$ with $V_2=0.0481771.$
\end{prop} 
\begin{proof}
By definition,  the conditional optimal set of two-points is the conditional set $\gb$ itself, and the corresponding conditional quantization error is given by 
\[V_2=V_{1, 2, 1}=\frac{37}{768}=0.0481771.\]
Thus, the proposition is yielded. 
\end{proof} 
 
\begin{prop}
The conditional optimal set of three-points is the set $\ga_3=\set{\frac{1}{4},\frac{1}{2},\frac{5}{6}}$ with $V_3=0.01114.$
\end{prop} 
\begin{proof}
By Equation~\eqref{eqMe2}, we see that 
\[V_{2,2,1}=0.0435475, \, V_{1,3,1}=0.0472005, \te{ and } V_{1,2,2}=0.01114.\]
Since $V_{1,2,2}$ is minimum among all the above possible errors, we can deduce that $k=1$, $\ell=2$, and $m=2$. Hence, by \eqref{eqMe1}, we obtain the conditional optimal set of three-points as $\ga_3=\set{\frac{1}{4},\frac{1}{2},\frac{5}{6}}$ with $V_3=0.01114.$
\end{proof} 
\begin{prop}
The conditional optimal set of four-points is the set $\ga_4=\set{\frac{1}{12},\frac{1}{4},\frac{1}{2},\frac{5}{6}}$ with $V_4=0.00651042.$
\end{prop} 
\begin{proof}
Considering all possible errors $V_{i, j, k}$ we see that it is minimum when $i=2,\, j=2$ and $k=2$. Hence, using \eqref{eqMe1} and \eqref{eqMe2}, we deduce that  $\ga_4=\set{\frac{1}{12},\frac{1}{4},\frac{1}{2},\frac{5}{6}}$ with $V_4=0.00651042.$
 \end{proof} 

Proceeding in the similar way as the previous propositions, we can deduce  the following two propositions: 

\begin{prop}
The conditional optimal set of five-points is the set $\ga_5=\set{\frac{1}{12},\frac{1}{4},\frac{1}{2},\frac{7}{10},\frac{9}{10}}$ with $V_5=0.00354745.$
\end{prop} 
 \begin{prop}
The conditional optimal set of six-points is the set $\ga_6=\set{\frac{1}{12},\frac{1}{4},\frac{3}{8},\frac{1}{2},\frac{7}{10},\frac{9}{10}}$ with $V_6=0.00257089.$
\end{prop}

\begin{lemma} \label{lemma1} 
Let $n\in \D N$ be such that $n=4x+2$ for some $x\in \D N$. Let $\te{card}(\ga_n\ii[0, \frac 14])=k$, $\te{card}(\ga_n\ii [\frac 14, \frac 12])=\ell$, and $\te{card}(\ga_n\ii [\frac 12, 1])=m$. Then, $(k-1) : (\ell-2): (m-1)=1:1:2$. 
\end{lemma} 
\begin{proof}
Let $n=4x+2$ for some $x\in \D N$, and $k, \ell, m$ be the positive integers as defined in the hypothesis. Since $m=n+2-k-\ell=4x+4-k-\ell$, by \eqref{eqMe2}, we have 
\[V_{k, \ell, m}=\frac{1}{768} \left(\frac{32}{(-2 k-2 \ell+8 x+7)^2}+\frac{4}{(1-2 k)^2}+\frac{1}{(\ell-1)^2}\right),\]
which is minimum if  $k=x+1$ and $\ell=x+2$. Then, $m=2x+1$. Thus, we see that $(k-1) : (\ell-2): (m-1)=1: 1:2$, which is the lemma. 
\end{proof}
As a consequence of Lemma~\ref{lemma1}, we deduce the following corollary.  
\begin{cor}
Let $\ga_n$ be a conditional optimal set of $n$-points with $\te{card}(\ga_n\ii[0, \frac 14])=k$, $\te{card}(\ga_n\ii [\frac 14, \frac 12])=\ell$, and $\te{card}(\ga_n\ii [\frac 12, 1])=m$.  Then, for $n\geq 6$, we have $k, m\geq 1$ and $\ell\geq 2$. 
\end{cor} 

Let us now give the following theorem, which is the main theorem in this section. 
\begin{theorem} \label{theo1}
For $n\in \D N$ with $n\geq 6$, let $\ga_n$ be a conditional  optimal set of $n$-points for $P$. Let $\te{card}(\ga_n\ii[0, \frac 14])=k$, $\te{card}(\ga_n\ii [\frac 14, \frac 12])=\ell$, and $\te{card}(\ga_n\ii [\frac 12, 1])=m$. For some $x\in \D N$ if $n=4x+2$, then $(k, \ell, m)=(x+1, x+2, 2x+1)$; if $n=4x+3$, then $(k, \ell, m)=(x+1, x+2, 2x+2)$;  if $n=4x+4$, then $(k, \ell, m)=(x+2, x+2, 2x+2)$;  if $n=4x+5$, then $(k, \ell, m)=(x+2, x+2, 2x+3)$.  
\end{theorem}
\begin{proof}
By Lemma~\ref{lemma1}, it is known that if $n=4x+2$, then $(k, \ell, m)=(x+1, x+2, 2x+1)$. Using the similar technique that is used in Lemma~\ref{lemma1}, we can show that if $n=4x+3$, then $(k, \ell, m)=(x+1, x+2, 2x+2)$;  if $n=4x+4$, then $(k, \ell, m)=(x+2, x+2, 2x+2)$;  if $n=4x+5$, then $(k, \ell, m)=(x+2, x+2, 2x+3)$.  
Thus, the proof of the theorem is complete. 
\end{proof} 

\begin{note}
By Theorem~\ref{theo1}, for any given $n\geq 6$, we can easily calculate the values of $(k, \ell, m)$. Since the values of $(k, \ell, m)$ depend on $n$, writing $(k, \ell, m):=(k(n), \ell(n), m(n))$, we have 
\begin{align*} 
&\Big\{(k(n), \ell(n), m(n)\Big\}_{n=6}^\infty\\
&=\Big\{(2, 3, 3), (2, 3, 4), (3, 3,4), (3,3,5), (3, 4, 5), (3, 4, 6), (4, 4, 6), (4,4, 7), (4, 5, 7), (4, 5, 8), \cdots\Big\}.
\end{align*} 
Notice that if $n=4x+2$ for $x\in \D N$, then we have 
\begin{align*} 
&\Big\{(k(4x+2)-1, \ell(4x+2)-2, m(4x+2)-1\Big\}_{x=1}^\infty=\Big\{(1,1,2), (2,2,4), (3,3, 6), (4, 4, 8), \cdots\Big\} 
\end{align*} 
implying \[\Big\{(k(4x+2)-1, \ell(4x+2)-2, m(4x+2)-1\Big\}_{x=1}^\infty=\Big\{x(1, 1, 2) : x\in \D N\Big\}.\]
\end{note} 

\subsection{Conditional optimal sets of $n$-points and the $n$th conditional quantization errors} Let $n\geq 6$ be a positive integer. To determine the optimal sets of $n$-points and the $n$th conditional quantization errors, first using Theorem~\ref{theo1}, we determine the corresponding values of $k, \ell$, and $m$. Once $k, \ell, m$ are known, by using \eqref{eqMe1}, we calculate the sets 
$\ga_n\ii [0, \frac 14]$, $\ga_n\ii [\frac 14, \frac 12]$, and $\ga_n\ii [\frac 12, 1]$. Then, $\ga_n$ is given by 
\[\ga_n=(\ga_n\ii [0, \frac 14])\UU (\ga_n\ii [\frac 14, \frac 12])\UU (\ga_n\ii [\frac 12, 1]), \]
and the $n$th conditional quantization error is obtained by using the formula \eqref{eqMe2}. 
\qed
 
\begin{exam}
Let $n=59$, then as $n=4\times 14+3=4x+3$, where $x=14$, by Theorem~\ref{theo1}, we have $(k, \ell, m)= (x+1, x+2, 2x+2)=(15, 16, 30)$. Hence, by \eqref{eqMe1} and \eqref{eqMe2}, we have the $n$th conditional optimal set of $n$-points, for $n=56$ as  
\[\ga_{59}= \Big\{\frac{1}{116} (2 j-1) : 1\leq j\leq 15\Big\}\UU \Big\{\frac{j-1}{60}+\frac{1}{4} : 1\leq j\leq 16\Big\}\UU \Big\{\frac{j-1}{59}+\frac{1}{2} : 1\leq j\leq 30\Big\} \]
with $n$th conditional quantization error $V_{59}=V_{15, 16, 30}=\frac{12115621}{505875628800}$. 
\end{exam}

\begin{theorem}\label{theo3} 
The conditional quantization dimension $D(P)$ of the probability measure $P$ exists, and $D(P)=1$. 
\end{theorem}

\begin{proof}
For any $n\in \D N$ with $n\geq 6$, there exists a positive integer $x$ depending on $n$ such that $4x+2\leq n\leq 4(x+1)+2$. Then, 
$V_{x+2, x+3, 2x+3}\leq V_n\leq V_{x+1, x+2, 2x+1}$. By \eqref{eqMe2}, we see that $V_{x+2, x+3, 2x+3}\to 0$ and  $V_{x+1, x+2, 2x+1}\to 0 $ as $n\to \infty$, and so by squeeze theorem, $V_n\to 0$ as $n\to \infty$,
i.e., $V_\infty=0$.
We can take $n$ large enough so that $ V_{x+1, x+2, 2x+1}-V_\infty<1$. Then,
\[0<-\log (V_{x+1, x+2, 2x+1}-V_{\infty})\leq -\log (V_n-V_\infty)\leq -\log (V_{x+2, x+3, 2x+3}-V_\infty)\]
yielding
\[\frac{2 \log (4x+2)}{-\log (V_{x+2, x+3, 2x+3}-V_\infty)}\leq \frac{2\log n}{-\log(V_n-V_\infty)}\leq \frac{2 \log (4x+6)}{-\log (V_{x+1, x+2, 2x+1}-V_\infty)}.\]
Notice that
\begin{align*}
\lim_{n\to \infty}\frac{2 \log (4x+2)}{-\log (V_{x+2, x+3, 2x+3}-V_\infty)}=1, \te{ and } \lim_{n\to \infty}\frac{2 \log (4x+6)}{-\log (V_{x+1, x+2, 2x+1}-V_\infty)}=1.
\end{align*}
Hence, $\mathop{\lim}\limits_{n\to \infty}  \frac{2\log n}{-\log(V_n-V_\infty)}=1$, i.e., the conditional quantization dimension $D(P)$ of the probability measure $P$ exists and $D(P)=1$.
Thus, the proof of
the theorem is complete.
\end{proof}

\begin{theorem} \label{theo4} 
The $D(P)$-dimensional quantization coefficient for $P$ exists as a finite positive number and equals $\frac 1{12}$.
\end{theorem}
\begin{proof}
For any $n\in \D N$ with $n\geq 6$, there exists a positive integer $x$ depending on $n$ such that $4x+2\leq n\leq 4(x+1)+2$. Then,   
$V_{x+2, x+3, 2x+3}\leq V_n\leq V_{x+1, x+2, 2x+1}$ and  $V_\infty=0$. Since
\begin{align*}
\lim_{n\to \infty} n^2 (V_n-V_\infty)&\geq \lim_{n\to \infty} (4x+2)^2 (V_{x+2, x+3, 2x+3}-V_\infty)=\frac1{12}, \te{ and } \\
\lim_{n\to \infty} n^2 (V_n-V_\infty)&\leq \lim_{n\to \infty} (4x+6)^2 (V_{x+1, x+2, 2x+1}-V_\infty)=\frac 1{12},
\end{align*}
by squeeze theorem, we have $\mathop{\lim}\limits_{n\to \infty} n^2 (V_n-V_\infty)=\frac 1{12}$, which is the theorem. 
\end{proof}

\section{Conditional optimal sets of $n$-points and the $n$th conditional quantization errors with $(k-1)$ interior elements and one boundary element in the conditional set for all $n\geq k$ on a unit line segment} \label{sec4} 

In this section, for the uniform distribution $P$ on the line segment $[0, 1]$ with respect to the conditional set $\gb:=\set{\frac 1k, \frac 2k, \cdots, \frac{k-1}{k}, \frac k k}$, we calculate the conditional optimal sets of $n$-points and the $n$th conditional quantization errors for all $n\in \D N$ with $n\geq k$. 
Let $\ga_n$ be a conditional optimal  set of $n$-points with the $n$th conditional quantization error $V_n$, where $n\in \D N$ with $n\geq k$. 
Write 
\begin{equation} \label{eq90} J_{k, j}:=[\frac {j-1}{k}, \frac j k] \te{ and } \te{card}(\ga_n\ii J_{k, j})=n_j \te{ for } 1\leq j\leq k.\end{equation} 
Notice that $n_j$ satisfies: $n_1\geq 1$, $n_j\geq 2$ for $2\leq j\leq k$. 
 By Proposition~\ref{prop0}, we know that
\begin{equation}\label{eqMe3}  \ga_n\ii J_{k, 1}= \Big\{\frac{2j-1}{k(2n_1-1)} : 1\leq j\leq n_1\Big\} \te{ with } V(P; \set{\ga_n\ii J_{k, 1}, J_{k, 1}})=\frac 1{3k^3(2n_1-1)^2},
\end{equation} and
\begin{equation} \label{eqMe4}
 \ga_n\ii J_{k, j}=\Big\{\frac {j-1}{k}+\frac{q-1}{k(n_j-1)}  : 1\leq q\leq n_j\Big\}  \te{ with } V(P; \set{\ga_n\ii J_{k, j}, J_{k, j}})=\frac 1{12k^3(n_j-1)^2} 
 \end{equation} 
 for $2\leq j\leq k$.
Notice that 
 \begin{equation} \label{eqMe5}  \ga_n=\UU_{j=1}^k \ga_n\ii J_{k,j} \te{ with }  
  V_n:=V_{n_1, n_2, \cdots, n_k}(P)= \sum_{j=1}^k V(P; \set{\ga_n\ii J_{k, j}, J_{k, j}}).
\end{equation} 
\begin{prop} \label{propMe0} 
The optimal set of $k$-points is the set $\gb=\set{\frac j k : 1\leq j\leq k}$ with $V_k=\frac{k+3}{12 k^3}.$
\end{prop} 
\begin{proof}
By definition,  the conditional optimal  set of $k$-points is the conditional set $\gb$ itself, and the corresponding conditional quantization error is given by 
\[V_k=\sum_{j=1}^k V(P; \set{\ga_n\ii J_{k, j}, J_{k, j}})=V(P; \set{\set{\frac 1k}, J_{k, 1}})+(k-1)V(P; \set{\set{\frac 1k, \frac 2 k}, J_{k, 2}})=\frac{k+3}{12 k^3}.\]
Thus, the proposition is yielded. 
\end{proof}

\begin{lemma} \label{lemma2} 
Let $n\in \D N$ be such $n\geq k$. Let $n_j$ be the positive integers as defined by \eqref{eq90}. Then, for $2\leq i<j\leq k$, $|n_i-n_j|=0 \te{ or } 1$. 
\end{lemma} 
\begin{proof}
Recall that for $2\leq i<j\leq k$, $n_i+n_j\geq 4$. 
Let us first assume that $n_i+n_j$ is an even number, i.e., $n_i+n_j=2m$ for some $m\geq 2$. Then, 
\[V(P; \set{\ga_n\ii J_{k, i}, J_{k, i}})+V(P; \set{\ga_n\ii J_{k, j}, J_{k, j}})=\frac 1{12k^3}\Big(\frac 1{(n_i-1)^2}+\frac 1{(n_j-1)^2}\Big).\]
By routine, we see that the above expression is minimum if $n_1=n_2=m$. Similarly, if $n_i+n_j=2m+1$ for some $m\geq 2$, then we see that the above expression is minimum if $(n_i, n_j)=(m, m+1)$, or  $(n_i, n_j)=(m+1, m)$. This yields the fact that for $2\leq i<j\leq k$, $|n_i-n_j|=0 \te{ or } 1$, which is the lemma. 
\end{proof}

\begin{lemma} \label{lemma3} 
Let $n\in \D N$ be such $n\geq k$. Let $n_j$ be the positive integers as defined by \eqref{eq90}. Then, for $2\leq j\leq k$, $|n_1-n_j|=0 \te{ or } 1$ with $n_1\leq n_j$. 
\end{lemma} 
\begin{proof}
Recall that $n_1\geq 1$ and for $2\leq j\leq k$, we have $n_j\geq 2$. 
Let us first assume that $n_1+n_j$ is an even number, i.e., $n_1+n_j=2m$, i.e., $n_j=2m-n_1$ for some $m\in \D N$ with $m\geq 2$. Then, 
\[V(P; \set{\ga_n\ii J_{k, 1}, J_{k, 1}})+V(P; \set{\ga_n\ii J_{k, j}, J_{k, j}})=\frac 1{3k^3}\Big(\frac 1{(2n_1-1)^2}+\frac 1{4(2m-n_1-1)^2}\Big).\]
By routine, we see that the above expression is minimum if $n_1=n_2=m$. Similarly, if $n_1+n_j=2m+1$ for some $m\in \D N$, then we see that the above expression is minimum if $(n_1, n_j)=(m, m+1)$. Thus, for $2\leq j\leq k$, we have $|n_1-n_j|=0 \te{ or } 1$ with $n_1\leq n_j$, which is the lemma. 
\end{proof}
Let us now give the following theorem, which is the main theorem is this section. This theorem helps us to determine the conditional optimal sets of $n$-points and the $n$th conditional quantization errors for all $n\in \D N$ with $n\geq k$. 

\begin{theorem} \label{Theo2}
For $n\geq k$, let $\ga_n$ be a conditional optimal  set of $n$-points such that $n=mk+\ell$ some $\ell, m\in \D N$ and $0\leq\ell<k$. Then,
\par 
$(i)$  if $\ell=0$, then 
$\te{card}(\ga_n\ii J_{k, 1})=m \te{ and } \te{card}(\ga_n\ii J_{k, j})=m+1 \te{ for } 2\leq j\leq k$; 
\par
$(ii)$ if $1\leq \ell<k$, then 
$\te{card}(\ga_n\ii J_{k, 1})=m+1 \te{ and } \te{card}(\ga_n\ii J_{k, j})=m+2 \te{ for } j\in \set{j_1, j_2, \cdots, j_{\ell-1}}, $ and $\te{card}(\ga_n\ii J_{k, j})=m+1$ for $j\in \set{2, 3, \cdots, k}\setminus \set{j_1, j_2, \cdots, j_{\ell-1}}$,
where $\set{j_1, j_2, \cdots, j_{\ell-1}}$ is any subset of $\ell-1$ elements of the set $\set{2, 3, \cdots, k}$. 
\end{theorem}
 \begin{proof}
 The proof follows as a consequence of Lemma~\ref{lemma2} and Lemma~\ref{lemma3}. 
 \end{proof}

\begin{remark}
Notice that in $(i)$ of Theorem~\ref{Theo2}, we have 
$\sum_{j=1}^k \te{card}(\ga_n\ii J_{k, j})=mk+(k-1)$, on the other hand, in $(ii)$ of Theorem~\ref{Theo2}, we have $\sum_{j=1}^k \te{card}(\ga_n\ii J_{k, j})=mk+\ell+(k-1)$, i.e., in the sum an extra term $(k-1)$ occurs. This happens because in the conditional optimal  set of $n$-points, $(k-1)$ elements from the conditional set are counted two times. 
\end{remark} 

\subsection{conditional optimal  sets of $n$-points and the $n$th conditional quantization errors} Let $n\geq k$ be a positive integer. To determine the optimal sets of $n$-points and the $n$th conditional quantization errors, first using Theorem~\ref{Theo2}, we determine the values of $n_j$, where $n_j=\te{card}(\ga_n\ii J_{k, j})$. Once $n_j$ are known by using the formulae given in \eqref{eqMe3} and \eqref{eqMe4}, we calculate the sets 
$\ga_n\ii J_{k, j}$ and the corresponding distortion errors $V(P; \set{\ga_n\ii J_{k, j}, J_{k, j}})$ for all $1\leq j\leq k$. Then, using the expressions in \eqref{eqMe5}, we obtain the conditional optimal  set $\ga_n$ and the corresponding $n$th conditional quantization error $V_n$. As an illustration, see Example~\ref{exam1} given below. 
\qed
 
\begin{exam}\label{exam1}
Let $P$ be the uniform distribution on the closed interval $[0, 1]$. 
Choose $k=5$, i.e., the conditional set is $\gb:=\set{\frac 15, \frac 25, \frac 35, \frac 45, 1}$. Then, the optimal set of $n$-points for any $n\geq 5$ exists. Notice that by Proposition~\ref{propMe0}, the conditional optimal  set of five-points is the conditional set $\gb$ with  the conditional quantization error 
\[V_5=\frac{k+3}{12 k^3}=\frac{2}{375}.\]
To determine a conditional optimal  set of $n$-points, for some $n$, $n=19$ say, we proceed as follows: 
\par 
We have $n=19=3\times 5+4$, i.e., we have $m=3$ and $\ell=4$. Recall Theorem~\ref{Theo2} $(ii)$. Let $\te{card}(\ga_n\ii J_{k, j})=n_j$ for $1\leq j\leq 5$. Choose any $\set{j_1, j_2, j_3} \ci \set{2, 3, 4, 5}$. Let $\set{j_1, j_2, j_3}=\set{2, 4, 5}$. Then, $\set{2, 3, 4, 5}\setminus \set{j_1, j_2, j_3}=\set{3}$ yielding $n_1=4$, $n_2=n_4=n_5=5$, and $n_3=4$. Then, using \eqref{eqMe3} and \eqref{eqMe4}, we have 
\begin{align*}
\ga_n\ii J_{k, 1}&=\Big\{\frac {2j-1}{35} : 1\leq j\leq 4\Big\}=\Big\{\frac{1}{35},\frac{3}{35},\frac{1}{7},\frac{1}{5}\Big\} \te{ with } V(P; \set{\ga_n\ii J_{k, 1}, J_{k, 1}})=\frac{1}{18375}, \\
\ga_n\ii J_{k, 2}&=\Big\{\frac{1}{5}+\frac{q-1}{20} : 1\leq q\leq 5\Big\}=\Big\{\frac{1}{5},\frac{1}{4},\frac{3}{10},\frac{7}{20},\frac{2}{5}\Big\} \te{ with } V(P; \set{\ga_n\ii J_{k, 2}, J_{k, 2}})=\frac{1}{24000},  \\
\ga_n\ii J_{k, 3}&=\Big\{\frac{2}{5}+\frac{q-1}{15} : 1\leq q\leq 4\Big\}=\Big\{\frac{2}{5},\frac{7}{15},\frac{8}{15},\frac{3}{5}\Big\} \te{ with } V(P; \set{\ga_n\ii J_{k, 3}, J_{k, 3}})=\frac{1}{13500},  \\
\ga_n\ii J_{k, 4}&=\Big\{\frac{3}{5}+\frac{q-1}{20} : 1\leq q\leq 5\Big\}=\Big\{\frac{3}{5},\frac{13}{20},\frac{7}{10},\frac{3}{4},\frac{4}{5}\Big\} \te{ with } V(P; \set{\ga_n\ii J_{k, 4}, J_{k, 4}})=\frac{1}{24000},  \\
\ga_n\ii J_{k, 5}&=\Big\{\frac{4}{5}+\frac{q-1}{20} : 1\leq q\leq 5\Big\}=\Big\{\frac{4}{5},\frac{17}{20},\frac{9}{10},\frac{19}{20},1\Big\} \te{ with } V(P; \set{\ga_n\ii J_{k, 5}, J_{k, 5}})=\frac{1}{24000}.
\end{align*} 
Hence, using the expressions in \eqref{eqMe5}, we obtain 
\begin{align*}
\ga_n&=\Big\{ \frac{1}{35},\frac{3}{35},\frac{1}{7},\frac{1}{5}, \frac{1}{4},\frac{3}{10},\frac{7}{20},\frac{2}{5}, \frac{7}{15},\frac{8}{15},\frac{3}{5}, \frac{13}{20},\frac{7}{10},\frac{3}{4},\frac{4}{5}, \frac{17}{20},\frac{9}{10},\frac{19}{20},1\Big\} \te{ with } \\
V_n&=\sum_{j=1}^5 V(P; \set{\ga_n\ii J_{k, j}, J_{k, j}})=\frac{2683}{10584000}.
\end{align*} 
\end{exam} 

\section{conditional quantization for uniform distributions on the boundaries of regular polygons inscribed in a unit circle} \label{sec5}

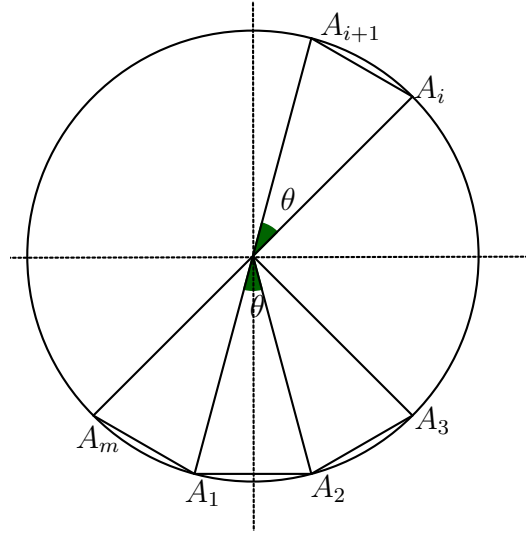
\begin{figure}
\begin{tikzpicture}[line cap=round,line join=round,>=triangle 45,x=1.0cm,y=1.0cm]
 \clip(-3.228655146506382,-3.654410217881296) rectangle (3.62483846731779,3.3884147257700925);
\draw [shift={(0.,0.)},line width=0.8 pt,color=qqwuqq,fill=qqwuqq,fill opacity=0.10000000149011612] (0,0) -- (-104.99998504172162:0.4507888805409463) arc (-104.99998504172162:-75.0000149582784:0.4507888805409463) -- cycle;
\draw [shift={(0.,0.)},line width=0.8 pt,color=qqwuqq,fill=qqwuqq,fill opacity=0.10000000149011612] (0,0) -- (45.:0.4507888805409463) arc (45.:75.00001495827838:0.4507888805409463) -- cycle;
\draw [line width=0.8 pt] (0.,0.) circle (3.cm);
\draw [line width=0.8 pt] (-0.776457,-2.89778)-- (0.776457,-2.89778);
\draw [line width=0.8 pt] (0.,0.)-- (-0.776457,-2.89778);
\draw [line width=0.8 pt] (0.,0.)-- (0.776457,-2.89778);
\draw [line width=0.8 pt] (0.776457,-2.89778)-- (2.12132,-2.12132);
\draw [line width=0.8 pt] (0.,0.)-- (2.12132,-2.12132);
\draw [line width=0.8 pt] (0.,0.)-- (-2.12132,-2.12132);
\draw [line width=0.8 pt] (-2.12132,-2.12132)-- (-0.776457,-2.89778);
\draw [line width=0.8 pt] (2.12132,2.12132)-- (0.776457,2.89778);
\draw [line width=0.8 pt] (0.,0.)-- (2.12132,2.12132);
\draw [line width=0.8 pt] (0.,0.)-- (0.776457,2.89778);
\draw (-2.4983771600300546,-2.116025544703235) node[anchor=north west] {$A_m$};
\draw (-1.0506686701728059,-2.8329977460556018) node[anchor=north west] {$A_1$};
\draw (0.6225394440270424,-2.7728925619834753) node[anchor=north west] {$A_2$};
\draw (2.0050112697220075,-1.871314800901582) node[anchor=north west] {$A_3$};
\draw (1.9950112697220075,2.531915852742294) node[anchor=north west] {$A_i$};
\draw (0.7478287002253894,3.36820435762584) node[anchor=north west] {$A_{i+1}$};
\draw (-0.1841697971450079,-0.36815927873779586) node[anchor=north west] {$\theta$};
\draw (0.2213824192336545,1.0541547708489808) node[anchor=north west] {$\theta$};
\draw [line width=0.8 pt,dash pattern=on 1pt off 1pt] (0.,4.)-- (0.011014274981214621,-3.779128474830959);
\draw [line width=0.8 pt,dash pattern=on 1pt off 1pt] (-3.3999549211119455,-0.022554470323070133)-- (3.707483095416974,-0.007528174305038579);
\end{tikzpicture}
\vspace{-0.1 in}
 \caption{The regular $m$-sided polygon inscribed in a unit circle.} \label{Fig1}
\end{figure}

Let the equation of the unit circle be $x_1^2+x_2^2=1$. Let $A_1A_2\cdots A_m$ be a regular $m$-sided polygon for some $m\geq 3$ inscribed in the circle as shown in Figure~\ref{Fig1}.  Let $\ell$ be the length of each side. Then,  the length of the boundary of the polygon is given by $\ell m$. Let $P$ be the uniform distribution defined on the boundary of the polygon. Then, the probability density function (pdf) $f$ for the uniform distribution $P$ is given by $f(x_1, x_2)=\frac 1{m\ell}$ for all $(x_1, x_2)\in A_1A_2\cdots A_m$, and zero otherwise. Let $\gq$ be the central angle subtended by each side of the polygon. Then, we know $\gq=\frac{2\pi}m$. Let the polar angles of the vertices $A_j$ of the polygon be given by $\gq_j$, where $1\leq j\leq m$. 
 Without any loss of generality, due to rotational symmetry, we can always assume that the side $A_1A_2$ of the polygon is parallel to the $x_1$-axis, as shown in Figure~\ref{Fig1}. Then, we have 
 \[\gq_1=\frac {3\pi}2-\frac {\gq} 2 =\frac {3\pi}2-\frac \pi m \te{ and } \gq_j=\gq_1+(j-1)\frac{2\pi}m \te{ for } 2\leq j\leq m.\]
Let $\gb$ be the set of all vertices of the polygon, i.e., 
\[\gb:= \set{(\cos\gq_j, \sin \gq_j) : 1\leq j\leq m}.\]
Notice that the Cartesian coordinates of the vertices $A_1$ and $A_2$ are given by, respectively, $(-\sin \frac{\pi}m, -\cos \frac{\pi}m)$ and $(\sin \frac{\pi}m, -\cos \frac{\pi}m)$. Hence, 
\[A_1A_2=\set{(t, -\cos \frac{\pi}m) : -\sin \frac{\pi}m\leq t\leq \sin \frac{\pi}m}.\]
Moreover, the length $\ell$ of each side is given by $\ell=2 \sin \frac{\pi}m.$
Let $\ga_n$ be a conditional optimal  set of $n$-points for $P$ with respect to the conditional set $\gb$, i.e., $\ga_n$ exists for all $n\geq m$.  
Let 
\begin{equation} \label{eqMe45} \te{card}(\ga_n\ii A_iA_{i+1})=n_i \te{ where } 1\leq i\leq m \te{ and } A_{m+1} \te{ is identified as } A_1.
\end{equation}  
Then, notice that 
\[n_i\geq 2 \te{ for all } 1\leq i\leq m \te{ and } n_1+n_2+\cdots +n_m=n+m,\]
as each of the vertices are counted two times. 

\begin{prop} \label{prop01} 
Let $P$ be the uniform distribution defined on the boundary of the regular $m$-sided polygon inscribed in the unit circle. Let $\te{card}(\ga_n\ii A_1A_2)=n_1$. Then, 
\begin{align}\label{eqMe44} \ga_n\ii A_1A_2&=\Big\{\Big(-\sin \frac{\pi }{m}+\frac{2 (j-1) \sin \frac{\pi }{m}}{n_1-1}, -\cos\frac{\pi }{m}\Big) : 1\leq j\leq n_1\Big\} 
 \end{align} 
with the corresponding distortion error 
\begin{equation} \label{eqMe451} V(P; \set{\ga_n\ii A_1A_2, A_1A_2})=\frac{\sin ^2\frac{\pi }{m}}{3 m (n_1-1){}^2}.
\end{equation} 
\end{prop} 

\begin{proof} 
Notice that the line segment $A_1A_2$ is parallel to the $x_1$-axis and lies on the line $x_2=-\cos\frac{\pi }{m}$. Hence, replacing $c$ by $(-\sin \frac{\pi}m, -\cos\frac{\pi} m)$ and $d$ by $(\sin \frac{\pi}m, -\cos\frac{\pi} m)$, by Proposition~\ref{prop0}, we obtain 
\[\ga_n\ii A_1A_2=\Big\{\Big(c_j, -\cos\frac \pi m\Big) : 1\leq j\leq n_1\Big\}, \te{ where } c_j=-\sin\frac{\pi} m+\frac{2(j-1)\sin\frac{\pi} m}{m}.\]
Recall $\ell= 2\sin \frac{\pi}m$. Hence, 
\begin{align*}
V(P; \set{\ga_n\ii A_1A_2, A_1A_2})&=\frac 1 {m\ell} \Big(2 \int_{c_1}^{\frac{1}{2} (c_1+c_2)} \rho((t, -\cos \frac{\pi}m), (c_1, -\cos \frac{\pi}m))  \, dt\\
&\qquad +(n_1-2) \int_{\frac{1}{2} (c_1+c_2)}^{\frac{1}{2} (c_2+c_3)} \rho((t, -\cos \frac{\pi}m), (c_2, -\cos \frac{\pi}m)) \, dt \Big) \\
&=\frac{\sin ^2\frac{\pi }{m}}{3 m (n_1-1){}^2},
\end{align*} 
which yields the proposition. 
\end{proof} 

The following lemma, which is similar to Lemma~\ref{lemma2}, is also true here. 
\begin{lemma} \label{lemma315} 
Let $n\in \D N$ be such $n\geq m$. Let $n_i$ be the positive integers as defined by \eqref{eqMe45}. Then, for $1\leq i\neq  j\leq m$, $|n_i-n_j|=0 \te{ or } 1$. 
\end{lemma} 

Let $T: \D R^2 \to \D R^2$ be an affine transformations such that for all $(x, y)\in \D R^2$, we have 
\[T(x, y)=(a x+b y, c x+d y),\]
where 
\begin{align*}
a&=\frac{1}{2}\Big(\sin \frac{3 \pi }{m} \csc \frac{\pi }{m}-1\Big), \, b= \frac{1}{2} \Big(-\sin \frac{3 \pi }{m}\sec  \frac{\pi }{m} -\tan \frac{\pi }{m}\Big),\\
c&= \frac{1}{2} \Big(\cot \frac{\pi }{m}-\cos \frac{3 \pi }{m} \csc \frac{\pi }{m}\Big), \te{ and } d= \frac{1}{2} \Big(\cos \frac{3 \pi }{m} \sec \frac{\pi }{m}+1\Big).
\end{align*} 
Also, for any $j\in \D N$, by $T^j$ it is meant the composition mapping 
$T^j=T\circ T\circ T\circ\cdots j\te{-times}.$ If $j=0$, i.e., by $T^0$ it is meant the identity mapping on $\D R^2$. Then, notice that 
\begin{align*} T^{i-1}(A_1A_{2})&=A_{i}A_{i+1} \te{ for } 1\leq i\leq m, \te{ where } A_{m+1} \te{ is identified as } A_1.  
\end{align*}
Let us now give the following theorem, which is the main theorem is this section. This theorem helps us to determine the conditional optimal  sets of $n$-points and the $n$th conditional quantization errors for all $n\in \D N$ with $n\geq m$. 

\begin{theorem} \label{Theo21}
For $n\geq m$, let $\ga_n$ be a conditional optimal  set of $n$-points such that $n=mk+\ell$ for some $k, \ell \in \D N$ and $0\leq\ell<m$. Then, identifying $A_{m+1}$ by $A_1$, we have
\par 
$(i)$  if $\ell=0$, then 
$\te{card}(\ga_n\ii A_iA_{i+1})=k+1  \te{ for } 1\leq i\leq m$;
\par
$(ii)$ if $1\leq \ell<m$, then 
$\te{card}(\ga_n\ii A_iA_{i+1})=k+2 \te{ for } i\in \set{i_1, i_2, \cdots, i_{\ell}}$ and $\te{card}(\ga_n\ii A_iA_{i+1})=k+1$ for $i\in \set{1, 2,  \cdots, m}\setminus \set{i_1, i_2, \cdots, i_{\ell}}$,
where $\set{i_1, i_2, \cdots, i_{\ell}}$ is any subset of $\ell$ elements of the set $\set{1, 2,  \cdots, m}$. 
\end{theorem}
 \begin{proof}
 The proof follows as a consequence of Lemma~\ref{lemma315}. 
 \end{proof}

\subsection{Conditional optimal  sets of $n$-points and the $n$th conditional quantization errors} Let $n\geq m$ be a positive integer. To determine the conditional optimal sets of $n$-points and the $n$th conditional quantization errors, first using Theorem~\ref{Theo21}, we determine the values of $n_i$, where $n_i=\te{card}(\ga_n\ii A_i A_{i+1})$ and $A_{m+1}$ is identified as $A_1$. 
Recall Proposition~\ref{prop01}. For each $n_i$ assume that $\te{card}(\ga_n\ii A_1A_2)=n_i$, and calculate 
$\ga_n\ii A_1 A_2$ and $V(P; \set{\ga_n\ii A_1A_{2}, A_1A_2})$, denote them by  $\ga_n\ii A_1 A_2(n_i)$ and $V(P; \set{\ga_n\ii A_1A_{2}, A_1A_2})(n_i)$, respectively. Now, recall the affine transformation. Since the affine transformation, considered in this section, preserves the length, the distortion errors do not change under the affine transformation. Hence, for each $n_i$, we obtain $\ga_n\ii A_iA_{i+1}$ and $V(P; \set{\ga_n\ii A_iA_{i+1}, A_iA_{i+1}})$ as follows: 
\begin{align*} &\ga_n\ii A_iA_{i+1} =T^{i-1}\Big(\ga_n\ii A_1 A_2(n_i)\Big), \te{ and }\\
 V(P;  \set{\ga_n &\ii A_iA_{i+1}, A_iA_{i+1}})  =V(P; \set{\ga_n\ii A_1A_{2}, A_1A_{2}})(n_i).
 \end{align*} 
Once $\ga_n\ii A_iA_{i+1}$ and $ V(P; \set{\ga_n\ii A_iA_{i+1}, A_iA_{i+1}})$ are obtained, we calculate the conditional optimal  sets $\ga_n$ and the $n$th conditional quantization errors  using the following formulae: 
\[\ga_n=\UU_{i=1}^m T^{i-1}\Big(\ga_n\ii A_1 A_2(n_i)\Big)=\UU_{i=1}^mT^{i-1}\Big\{\Big(-\sin \frac{\pi }{m}+\frac{2 (j-1) \sin \frac{\pi }{m}}{n_i-1}, -\cos\frac{\pi }{m}\Big) : 1\leq j\leq n_i\Big\}\] 
and   \[V_n=\sum_{i=1}^m V(P; \set{\ga_n\ii A_1A_{2}, A_1A_{2}})(n_i)=\sum_{i=1}^m\frac{\sin ^2\frac{\pi }{m}}{3 m (n_i-1){}^2}.  \]
\qed

\begin{remark}
Since the conditional quantization dimension is same as the quantization dimension (see \cite{PR4}), and it is well-known that the quantization dimension of an absolutely continuous probability measure equals the Euclidean dimension of the underlying space, we can assume that the conditional quantization dimension of $P$ is one, i.e., $D(P)=1$.
\end{remark}

Let us now give the following proposition.

\begin{prop} \label{propMe23} 
Let $\ga_n$ be an optimal set of $n$-points for $P$ such that $n=mk$, where $k\in \D N$. Then, 
\[V_n=\frac 1{3 k^2}\sin^2\frac{\pi }{m}.\]
\end{prop} 

\begin{proof} Let $n=mk$ for some $k\in \D N$.  Let $n_i$ be the positive integers as defined by \eqref{eqMe45}. Then, by Lemma~\ref{lemma315}, we can say that 
\[n_1=n_2=\cdots=n_m=k+1.\]
Notice that each $n_i$ equals $k+1$. It happens because $\ga_n$ contains $m$ distinct elements from each side, but in each $n_i$ both the end points are counted. 
Hence, by \eqref{eqMe451}, we have $V_n=\frac 1{3 k^2}\sin^2\frac{\pi }{m}$. Thus, the proof of the proposition is complete. 
 \end{proof}   
\begin{theorem}\label{theo45}
Let $P$ be the uniform distribution on the boundary of a regular $m$-sided polygon inscribed in a unit circle. Then, the conditional quantization coefficient for $P$ exists as a finite positive number and equals $\frac{1}{3} m^2 \sin ^2(\frac{\pi }{m})$, i.e.,
$\lim\limits_{n\to \infty} n^2 (V_n-V_\infty)=\frac{1}{3} m^2 \sin ^2(\frac{\pi }{m}).$
\end{theorem}

\begin{proof} Let $n\in \D N$ be such that $n\geq m$. Then, there exists a unique positive integer $\ell(n)\geq 2$ such that $m\ell(n)\leq n<m(\ell(n)+1)$. Then,
\begin{equation} \label{eq46} (m\ell(n))^2V_{m(\ell(n)+1)}<n^2 V_n<(m(\ell(n)+1))^2V_{m\ell(n)}.
\end{equation}
Recall Proposition~\ref{propMe23}. By squeeze theorem, we have $ V_\infty =\lim_{n\to \infty} V_n=0.$
Moreover, we have  
\begin{align*}
&\lim_{n \to \infty} (m\ell(n))^2(V_{m(\ell(n)+1)}-V_\infty)=\lim_{n\to \infty}  (m\ell(n))^2 \frac 1{3 (\ell(n)+1)^2}\sin^2\frac{\pi }{m} =\frac{1}{3} m^2 \sin ^2\frac{\pi }{m},\end{align*}
and
\begin{align*}
&\lim_{n \to \infty} (m(\ell(n)+1))^2(V_{m\ell(n)}-V_\infty)=\lim_{n\to \infty}  (m(\ell(n)+1))^2 \frac 1{3 (\ell(n))^2}\sin^2\frac{\pi }{m}=\frac{1}{3} m^2 \sin ^2\frac{\pi }{m},
\end{align*}
and hence, by \eqref{eq46}, using squeeze theorem, we have
$\mathop{\lim}\limits_{n\to\infty} n^2 (V_n-V_\infty)=\frac{1}{3} m^2 \sin ^2\frac{\pi }{m}$, i.e., the conditional quantization coefficient exists as a finite positive number which equals $\frac{1}{3} m^2 \sin ^2\frac{\pi }{m}$.
Thus, the proof of the theorem is complete.
\end{proof}

\begin{remark}
It is known that for an absolutely continuous probability measure, the quantization dimension equals the Euclidean dimension of the underlying space, and the quantization coefficient exists as a finite positive number (see \cite{BW}). Since the conditional quantization dimension is same as the quantization dimension, and the conditional quantization coefficient is same as the quantization coefficient (see \cite{PR4}), by Theorem~\ref{theo45}, we can conclude that the quantization coefficient for the uniform distribution defined on the boundary of a regular $m$-sided polygon inscribed in a unit circle is $\frac{1}{3} m^2 \sin ^2\frac{\pi }{m}$, which depends on $m$ and is an increasing function of $m$. Thus, we can conclude that for absolutely continuous probability measures given in an Euclidean space, the quantization dimensions remain constant and it is equal to the dimension of the underlying space, but the quantization coefficients can be different.    
\end{remark}  

Let us now conclude the paper with the following remark. 

\begin{remark}
Although the conditional quantization in this paper is investigated for uniform distributions on line segments and regular polygons, by using a similar technique or by giving a major overhaul of the technique given in this paper, interested researchers can explore them for any probability distribution defined on the boundary of any geometrical shape.
\end{remark} 

\subsection*{Acknowledgements} We would like to thank the anonymous referees for their valuable comments and
suggestions.


\begin{thebibliography}{9999}

\bibitem[AW]{AW} E.F. Abaya and G.L. Wise, \emph{Some remarks on the existence of optimal quantizers}, Statistics \& Probability Letters, Volume 2, Issue 6, December 1984, Pages 349-351.
\bibitem[BW]{BW} J.A. Bucklew and G.L. Wise, \emph{Multidimensional asymptotic quantization theory with $r$th power distortion measures}, IEEE Transactions on Information Theory, 1982, Vol. 28 Issue 2, 239-247.
    
    \bibitem[BCDRV]{BCDRV} P. Biteng, M. Caguiat, D. Deb,  M.K. Roychowdhury, and B. Villanueva, \emph{Constrained quantization for a uniform distribution}, Houston Journal of Mathematics, Volume 50, Number 1, 2024, Pages 121-142.

\bibitem[CR]{CR}  D. \c C\"omez and M.K. Roychowdhury, \emph{Quantization for uniform distributions on stretched Sierpi\'nski triangles}, Monatshefte f\"ur Mathematik, Volume 190, Issue 1, 79-100 (2019).

  \bibitem[DFG]{DFG} Q. Du, V. Faber and M. Gunzburger, \emph{Centroidal Voronoi Tessellations: Applications and Algorithms}, SIAM Review, Vol. 41, No. 4 (1999), pp. 637-676.

\bibitem[DR]{DR} C.P. Dettmann and M.K. Roychowdhury, \emph{Quantization for uniform distributions on equilateral triangles}, Real Analysis Exchange, Vol. 42(1), 2017, pp. 149-166.

\bibitem[GG]{GG} A. Gersho and R.M. Gray, \emph{Vector quantization and signal compression}, Kluwer Academy publishers: Boston, 1992.

\bibitem[GKL]{GKL}  R.M. Gray, J.C. Kieffer and Y. Linde, \emph{Locally optimal block quantizer design}, Information and Control, 45 (1980), pp. 178-198.

\bibitem[GL]{GL} S. Graf and H. Luschgy, \emph{Foundations of quantization for probability distributions}, Lecture Notes in Mathematics 1730, Springer, Berlin, 2000.


\bibitem[GL1]{GL1} A. Gy\"orgy and T. Linder, \emph{On the structure of optimal entropy-constrained scalar quantizers},  IEEE transactions on information theory, vol. 48, no. 2, February 2002.

\bibitem [GL2]{GL2} S. Graf and H. Luschgy, \emph{The Quantization of the Cantor Distribution}, Math. Nachr., 183 (1997), 113-133.


\bibitem[GL3]{GL3} S. Graf and H. Luschgy, \emph{Quantization for probability measures with respect to the geometric mean error}, Math. Proc. Camb. Phil. Soc. (2004), 136, 687-717. 



  \bibitem[GN]{GN} R.M. Gray and D.L. Neuhoff, \emph{Quantization}, IEEE Transactions on Information Theory, Vol. 44, No. 6, October 1998, 2325-2383.


 \bibitem[H]{H} J. Hutchinson, \emph{Fractals and self-similarity}, Indiana Univ. J., 30 (1981), 713-747.



 \bibitem[HMRT]{HMRT}  J. Hansen, I. Marquez, M.K. Roychowdhury, and E. Torres, \emph{Quantization coefficients for uniform distributions on the boundaries of regular polygons}, Statistics \& Probability Letters, Volume 173, June 2021, 109060.


\bibitem[HNPR]{HNPR} C. Hamilton, E. Nyanney, M. Pandey, and M.K. Roychowdhury, \emph{Conditional constrained and unconstrained quantization for a uniform distribution on a hexagon}, arXiv:2401.10987 [math.PR]. 
    

\bibitem[KNZ]{KNZ}  M. Kesseb\"ohmer, A. Niemann and S. Zhu, \emph{Quantization dimensions of compactly supported probability measures via R\'enyi dimensions}, Trans. Amer. Math. Soc. (2023). 

\bibitem[P]{P} D. Pollard, \emph{Quantization and the Method of $k$-Means}, IEEE Transactions on Information Theory, 28 (1982), 199-205.
  \bibitem[P1]{P1} K. P\"otzelberger, \emph{The quantization dimension of distributions}, Math. Proc. Cambridge Philos. Soc., 131 (2001), 507-519.

\bibitem[PR1]{PR1} M. Pandey and M.K. Roychowdhury, \emph{Constrained quantization for probability distributions}, arXiv:2305.11110 [math.PR].
\bibitem[PR2]{PR2} M. Pandey and M.K. Roychowdhury, \emph{Constrained quantization for the Cantor distribution}, J. Fractal Geom. 11 (2024), no. 3/4, pp. 319-341.
\bibitem[PR3]{PR3} M. Pandey and M.K. Roychowdhury, \emph{Constrained quantization for a uniform distribution with respect to a family of constraints}, arXiv:2309:11498 [math.PR].
\bibitem[PR4]{PR4} M. Pandey and M.K. Roychowdhury, \emph{Conditional constrained and unconstrained quantization for probability distributions}, arXiv:2312:02965 [math.PR].
\bibitem[PR5]{PR5} M. Pandey and M.K. Roychowdhury, \emph{Constrained quantization for the Cantor distribution with a family of constraints}, arXiv:2401.01958[math.DS].



\bibitem[RR]{RR} J. Rosenblatt and M.K. Roychowdhury, \emph{Uniform distributions on curves and quantization}, Commun. Korean Math. Soc. 38 (2023), No. 2, pp. 431-450.


\bibitem[R1]{R1} M.K. Roychowdhury, \emph{Quantization and centroidal Voronoi tessellations for probability measures on dyadic Cantor sets}, Journal of Fractal Geometry, 4 (2017), 127-146.

\bibitem [R2]{R2} M.K. Roychowdhury, \emph{Least upper bound of the exact formula for optimal quantization of some uniform Cantor distributions}, Discrete and Continuous Dynamical Systems- Series A, Volume 38, Number 9, September 2018, pp. 4555-4570.
\bibitem[R3]{R3} M.K. Roychowdhury, \emph{Optimal quantization for the Cantor distribution generated by infinite similitudes}, Israel Journal of Mathematics 231 (2019), 437-466.


\bibitem[Z1]{Z1} P.L. Zador, \emph{Asymptotic Quantization Error of Continuous Signals and the Quantization Dimension}, IEEE Transactions on Information Theory, 28 (1982), 139-149.

\bibitem[Z2]{Z2} R. Zam, \emph{Lattice Coding for Signals and Networks: A Structured Coding Approach to Quantization, Modulation, and Multiuser Information Theory}, Cambridge University Press, 2014.

\end{thebibliography}
\end{document}